\documentclass[a4paper]{amsart}
\numberwithin{equation}{section}
\newtheorem{theorem}{Theorem}[section]
\newtheorem{corollary}{Corollary}[section]

\newtheorem{lemma}{Lemma}[section]

\theoremstyle{remark}

\usepackage{hyperref}
\usepackage{color}
\usepackage{graphicx}
\usepackage{epstopdf}
\usepackage{subfigure}
\title[some applications of differential subordination]
 {some applications of differential subordination for certain starlike functions}



\begin{document}
\begin{abstract}
We consider the class $\mathcal{S}^*(q_c)$ of
normalized starlike functions $f$ analytic in the open unit disk
$|z|<1$ that satisfying the inequality
\begin{equation*}
   \left|\left(\frac{zf'(z)}{f(z)}\right)^2-1\right|<c \quad
   (0<c\leq1).
\end{equation*}
In this article, we present some subordination relations and these relations are then used to obtain some corollaries for some subclass of analytic functions.
\end{abstract}

\author[R. Kargar and L. Trojnar-Spelina] {R. Kargar and L. Trojnar-Spelina}
\address{Young Researchers and Elite Club,
Ardabil Branch, Islamic Azad University, Ardabil, Iran}
\email {rkargar1983@gmail.com {\rm (Rahim Kargar)}}

\address{ Department of Mathematics, Rzesz\'{o}w University of Technology,  Al. Powsta\'{n}c\'{o}w Warszawy 12, 35-959 Rzesz\'{o}w, Poland}
\email{lspelina@prz.edu.pl {\rm (Lucyna Trojnar-Spelina)}}
\maketitle

\section{Introduction}
This paper studies the class $\mathcal{A}$ of analytic functions
$f$ in $\mathbb{D}=\{z\in \mathbb{C}:|z|< 1\}$ normalized by
the condition $f(0)=f'(0)-1=0$. Also, we denote by $\mathcal{U}$, the class of all univalent (one--to--one) functions. A function $f$ analytic in $\mathbb{D}$ is said to be subordinate to an analytic
univalent function $g$ written as $f(z)\prec g(z)$ if $f(0)=g(0)$ and
$f(\mathbb{D})\subset g(\mathbb{D})$. Applying the Schwarz lemma to $(g^{-1}\circ f)(z)$, if $f(z)\prec g(z)$, then there exists an
analytic function $\omega$ on $\mathbb{D}$ with $\omega(0)=0$ and $|\omega(z)|<1$ such that $f(z)=g(\omega(z))$ for all $z\in\mathbb{D}$.

Further, we say that the function $f \in {\mathcal{A}}$ is starlike when it
maps the set $\mathbb{D}$ onto a starlike domain with respect to
the origin. The function $f \in {\mathcal{A}}$ is called convex
function when $f(\mathbb{D})$ is a convex set.

For $c \in(0, 1]$ we denote by $\mathcal{S}^*(q_c)$ the class of
analytic functions $f$ in the unit disc $\mathbb{D}$ satisfying the
condition
\begin{equation*}
   \left|\left(\frac{zf'(z)}{f(z)}\right)^2-1\right|<c \quad
   (z\in\mathbb{D}).
\end{equation*}
The class $\mathcal{S}^*(q_c)$ was introduced in \cite{SJ}. Moreover, the class
$\mathcal{S}^*(q_1)\equiv \mathcal{SL}^*$ was considered in
\cite{JS}. It is easy to see that $f\in\mathcal{S}^*(q_c)$,
$c\in(0, 1]$ if and only if it satisfies the differential
subordination
\begin{equation}\label{qc}
   \frac{zf'(z)}{f(z)}\prec \sqrt{1+cz}:=q_c(z)\quad
   (z\in\mathbb{D}),
\end{equation}
where the branch of the square root is chosen in order to
$\sqrt{1}=1$. There are many interesting subclass of starlike functions which have been defined by subordination, see for example \cite{KargarAnal., KargarComplex, KumarRavi, KO2011, Men2014, Men2015, RaiSok2015, SharmaAfrica, Sok2011, SokSta}.

Note here that the function $q_c(z)=\sqrt{1+cz}$
maps $\mathbb{D}$ onto a set bounded by Bernoulli lemniscate, i.e.
set of all points on the right half-plane such that the product of
the distances from each point to the focuses $-1$ and $1$ is less
than $c$:
\begin{equation}\label{ome}
  \Omega_c=\{w\in \mathbb{C}\,:\, {\rm Re}\,  w >0,\,|w^2-1|<c \}.
\end{equation}
\begin{figure}[!ht]
\centering
\subfigure[توضیح زیر شکل اوّل]{
\includegraphics[width=5cm]{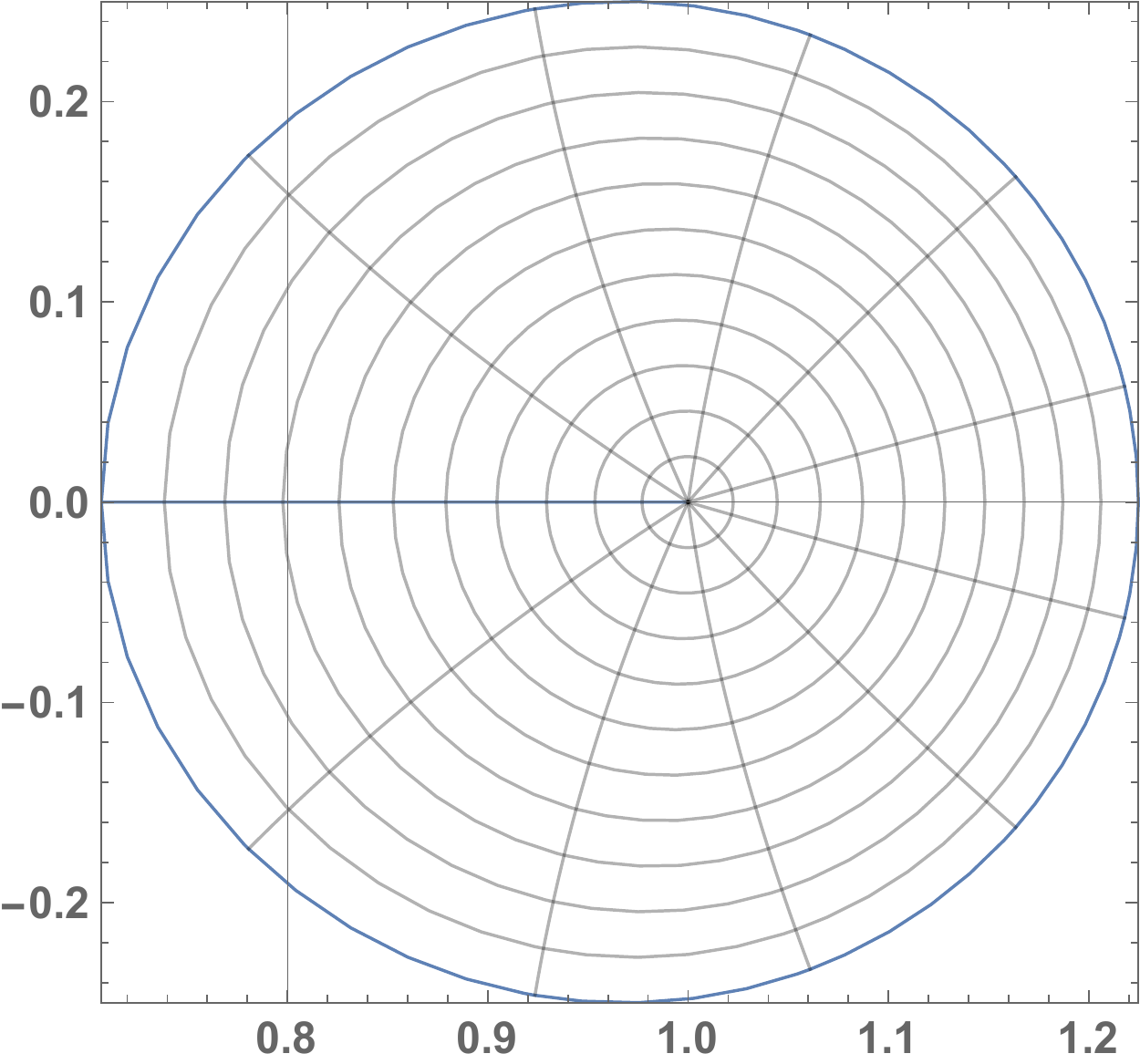}
	\label{fig:subfig1}
}
\hspace*{10mm}
\subfigure[توضیح زیر شکل دوّم]{
\includegraphics[width=5cm]{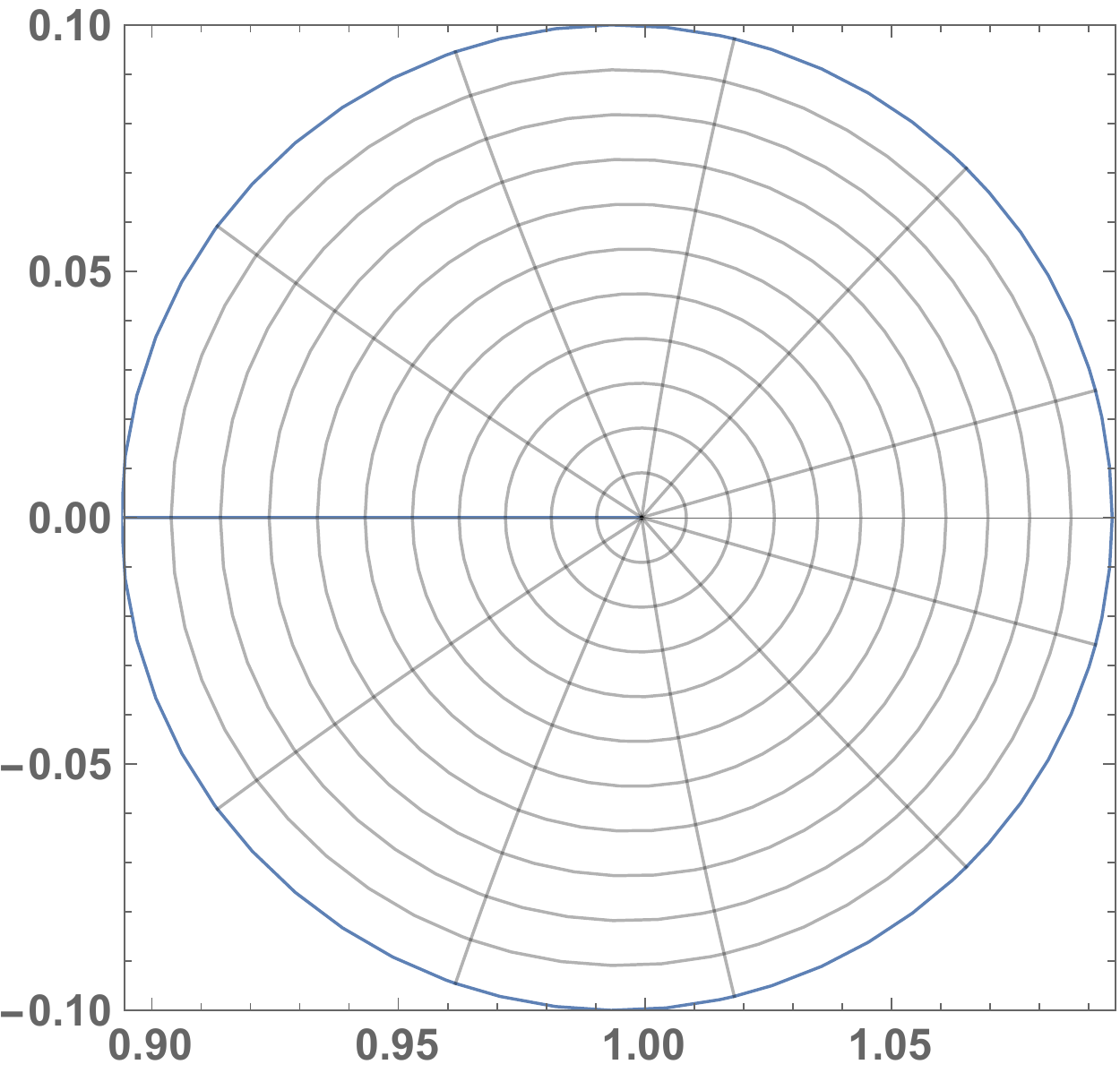}
	\label{fig:subfig2}
}

\caption[The boundary curve of $q_{1/2}(\Delta)$]
{\subref{fig:subfig1}: The boundary curve of $q_{1/2}(\mathbb{D})$  ،
 \subref{fig:subfig2}: The boundary curve of $q_{1/5}(\mathbb{D})$،
 }
\label{fig:subfig01}
\end{figure}

Since $q_c(\mathbb{D})=\Omega_c$ is a convex set (for example, see Fig. \ref{fig:subfig01}, for the cases $c=1/2$ and $1/5$), hence $q_c$ is convex and therefore $zq_c'(z)$ is starlike in $\mathbb{D}$.

Next, for the fixed constants $A$ and $B$ $(-1\leq B<A\leq 1)$ by
$\mathcal{S}^*[A,B]$ we denote the class of Janowski starlike
functions, introduced by Janowski \cite{WJ}, that consists of
functions $f\in\mathcal{ A}$ satisfying the condition
\begin{equation*}
   \frac{zf'(z)}{f(z)}\prec \frac{1+Az}{1+Bz}\quad
   (z\in\mathbb{D}).
\end{equation*}
We remark that $\mathcal{S}^*[1,-1]$ is the class $\mathcal{S}^*$
of starlike functions.

In order to prove one of our main results, we need the following lemma called Jack's Lemma.
\begin{lemma}\label{LemJack}
{\rm(}see \cite{Jack}, see also \cite[Lemma 1.3, p. 28]{SRU}{\rm)} Let $w$ be
a nonconstant function meromorphic in $\mathbb{D}$ with $w(0)=0$. If
\begin{equation*}
   |w(z_0)|=\max\{|w(z)|: |z|\leq |z_0|\}\quad (z\in\mathbb{D}),
\end{equation*}
then there exists a real number k $(k\geq1)$ such that
$z_0w'(z_0)=kw(z_0)$.
\end{lemma}
 In this paper, for analytic function $p(z)$ in the unit disk $\mathbb{D}$ we find some conditions that imply $p(z)\prec \sqrt{1+cz}$. Also, some interesting corollaries are obtained.
\section{Main Results}
The first result is the following.
\begin{theorem}\label{t21}
Assume that $p$ is an analytic function in $\mathbb{D}$ with
$p(0)=1$. Further assume that $|A|\leq1,$ $|B|<1,$ $ 0<c\leq1$ and
that $\gamma$ satisfies the following inequality
\begin{equation}\label{0Th1}
   \gamma\geq \frac{2(|A|+|B|)}{c(1-|B|)}(1+c).
\end{equation}
If $f$ satisfies the subordination
\begin{equation*}
    1+\gamma \frac{zp'(z)}{p(z)}\prec\frac{1+A z}{1+B z} \quad  (z\in\mathbb{D}),
\end{equation*}
then
\begin{equation*}
    p(z)\prec \sqrt{1+cz}\quad   (z\in\mathbb{D}).
\end{equation*}
\end{theorem}
\begin{proof}
We define the function $F$ as follows
\begin{equation}\label{F(z)}
    F(z):=1+\gamma \frac{zp'(z)}{p(z)} \quad {\rm{for}} \quad z \in \mathbb{D}
\end{equation}
and the function $w$ by the relation
\begin{equation}\label{w(z)}
    p(z)=\sqrt{1+cw(z)}.
\end{equation}
Since $p$ is an analytic function and $p(0)=1$, then $w$ is
meromorphic in $\mathbb{D}$ and $w(0)=0$. It suffices to show that
$|w(z)|<1$ in $\mathbb{D}$. From \eqref{w(z)} we have
\begin{equation*}
    \gamma \frac{zp'(z)}{p(z)}=\frac{c\gamma zw'(z)}{2(1+cw(z))},
\end{equation*}
and using this in \eqref{F(z)} we can express the function $F$ as
follows:
\begin{equation*}
    F(z)=1+\frac{c\gamma zw'(z)}{2(1+cw(z))}.
\end{equation*}
Therefore we have
\begin{equation*}
    \frac{F(z)-1}{A-B F(z)}=\frac{c\gamma zw'(z)}{2A(1+cw(z))-B[2(1+cw(z))+c\gamma
    zw'(z)]}.
\end{equation*}
Assume that there exists a point $z_0\in\mathbb{D}$ such that
\begin{equation*}
  \max_{|z|\leq |z_0|}|w(z)|=|w(z_0)|=1.
\end{equation*}
Then by Lemma \ref{LemJack}, there exists a number $k\geq1$ such
that $z_0w'(z_0) = kw(z_0)$. Let $w(z_0) = e^{i\theta}$. For this
$z_0$, we have
\begin{align*}
    \left|\frac{F(z_0)-1}{A-B F(z_0)}\right|&=\left|\frac{ck\gamma e^{i\theta}}{2A(1+ce^{i\theta})-B[2(1+ce^{i\theta})+c\gamma
    ke^{i\theta}]}\right|\\&\geq \frac{ck\gamma}{2|A||1+ce^{i\theta}|+|B||2+(2c+c\gamma
    k)e^{i\theta}|}\\&=\frac{ck\gamma}{2|A|\sqrt{1+2c\cos\theta+c^2}+|B|
    \sqrt{4+4(2c+c\gamma k)\cos\theta+(2c+c\gamma k)^2}}\\
    &=:H(\cos \theta).
\end{align*}
Now, define the function
\begin{equation}\label{H(t)}
  H(t)=\frac{ck\gamma}{2|A|\sqrt{1+2ct+c^2}+|B|\sqrt{4+4(2c+c\gamma k)t+(2c+c\gamma
  k)^2}}.
\end{equation}
A simple computation shows that $H'(t)< 0$. Thus the function $H$
is a decreasing function when $-1\leq t = \cos \theta\leq 1$.
Therefore
\begin{equation}\label{H(1)}
  H(t)\geq H(1)=\frac{ck\gamma}{2|A|(1+c)+|B|(2+2c+c\gamma k)}.
\end{equation}
Consider the function
\begin{equation}\label{L(k)}
  L(k)=\frac{ck\gamma}{2|A|(1+c)+|B|(2+2c+c\gamma k)} \quad k\geq1.
\end{equation}
It can be easily seen that $L'(k)>0$, thus we have
\begin{equation}\label{L(1)}
  L(k)\geq L(1)=\frac{c\gamma}{2|A|(1+c)+|B|(2+2c+c\gamma)}.
\end{equation}
Now from \eqref{H(t)}-\eqref{L(1)}, we have
\begin{equation*}
\left|\frac{F(z_0)-1}{A-B
F(z_0)}\right|\geq\frac{c\gamma}{2|A|(1+c)+|B|(2+2c+c\gamma)}.
\end{equation*}
Note that it follows from \eqref{0Th1} that the right hand side of
the above inequality is greater than or equal to $1$ but this is
contrary to the assumption $F(z)\prec (1+Az)/(1+Bz)$ and hence the
proof is completed.
\end{proof}

Taking $p(z)=zf'(z)/f(z)$ in Theorem \ref{t21}, we have the
following result:
\begin{corollary}\label{c21}
Let $|A|\leq1,$ $|B|<1,$ $0<c\leq1$
 and let
 \begin{equation*}
   \gamma\geq \frac{2(|A|+|B|)}{c(1-|B|)}(1+c).
\end{equation*}
If $f$ satisfies the following subordination
\begin{equation*}
1+\gamma\left(1+\frac{zf''(z)}{f'(z)}-\frac{zf'(z)}{f(z)}\right)\prec\frac{1+A
z}{1+B z}\quad   (z\in\mathbb{D}),
\end{equation*}
then $f\in\mathcal{S}^*(q_c)$.
\end{corollary}

Putting $c=1$ in Corollary \ref{c21}, we have:


\begin{corollary}\label{c22}
Let $|A|\leq1,$ $ B|<1$ and let
\begin{equation*}
   \gamma\geq \frac{4(|A|+|B|)}{1-|B|}.
\end{equation*}
If $f$ satisfies the following subordination
\begin{equation*}
1+\gamma\left(1+\frac{zf''(z)}{f'(z)}-\frac{zf'(z)}{f(z)}\right)\prec\frac{1+A
z}{1+B z}\quad   (z\in\mathbb{D}),
\end{equation*}
then $f\in\mathcal{SL}^*$.
\end{corollary}
If we take $A=1$ and $B=0$ in Corollary \ref{c22}, we obtain:

\begin{corollary}\label{c23}
Let $\gamma\geq 4$. If $f$ satisfies the following inequality
\begin{equation*}
{\rm Re}\left\{1+\gamma\left(1+\frac{zf''(z)}{f'(z)}-\frac{zf'(z)}{f(z)}\right)\right\}>0
\end{equation*}
for all  $z\in \mathbb{D}$ then $f\in\mathcal{SL}^*$.
\end{corollary}
Putting $p(z)=z\sqrt{f'(z)}/f(z)$ in Theorem \ref{t21}, we have:

\begin{corollary}\label{c24}
Let $|A|\leq1,$ $|B|<1,$ $0<c\leq1$ and let
\begin{equation*}
   \gamma\geq \frac{2(|A|+|B|)}{c(1-|B|)}(1+c).
\end{equation*}
If the function $f$ satisfies the following subordination
\begin{equation*}
1+\gamma\left(1+\frac{1}{2}\frac{zf''(z)}{f'(z)}-\frac{zf'(z)}{f(z)}\right)\prec\frac{1+A
z}{1+B z} \quad (z\in \mathbb{D}),
\end{equation*}
then
\begin{equation*}
\left|\left(\frac{z}{f(z)}\right)^2f'(z)-1\right|\leq c
\end{equation*}
for all $z \in \mathbb{D}.$
\end{corollary}

If we take $c=1$ in Corollary \ref{c24}, we have:
\begin{corollary}
Assume that $|A|\leq1,$ $|B|<1$ and that
\begin{equation*}
   \gamma\geq \frac{4(|A|+|B|)}{1-|B|}.
\end{equation*}
If
\begin{equation*}
1+\gamma\left(1+\frac{1}{2}\frac{zf''(z)}{f'(z)}-\frac{zf'(z)}{f(z)}\right)\prec\frac{1+A
z}{1+B z}\quad (z\in \mathbb{D}),
\end{equation*}
then $f$ is univalent in $\mathbb{D}$ by \cite[Theorem 2, p.
394]{OZ}.
\end{corollary}
By taking $p(z)=\sqrt{f'(z)}$ and $c=1$ in Theorem \ref{t21} we have
the following result:
\begin{corollary}
Assume that $|A|\leq1,$ $|B|<1$ and that
\begin{equation*}
   \gamma\geq \frac{4(|A|+|B|)}{1-|B|}.
\end{equation*}
If
\begin{equation*}
1+\gamma\left(\frac{1}{2}\frac{zf''(z)}{f'(z)}\right)\prec\frac{1+A
z}{1+B z}\quad (z\in \mathbb{D}),
\end{equation*}
then $f$ is univalent in $\mathbb{D}$ by \cite{OBPO}.
\end{corollary}
Assuming $p(z)=f(z)/z$ in Theorem \ref{t21} we obtain the following result.

\begin{corollary}\label{c27}
Let $|A|\leq1,$ $|B|<1,$ $0<c\leq1$ and let
\begin{equation*}
   \gamma\geq \frac{2(|A|+|B|)}{c(1-|B|)}(1+c).
\end{equation*}
If $f$ satisfies the following subordination
\begin{equation*}
    1+\gamma\left(\frac{zf'(z)}{f(z)}-1\right)\prec\frac{1+A z}{1+B z}\quad (z\in \mathbb{D}),
\end{equation*}
then
\begin{equation*}
   \left|\left(\frac{f(z)}{z}\right)^2-1\right|<c
\end{equation*}
for all $z\in \mathbb{D}.$
\end{corollary}

If we take $A=1$ and $B=0$ in Corollary \ref{c27}, then we have:
\begin{corollary}
Let $c \in (0,1]$ and let $ \gamma\geq 2(1+1/c)$. If $f$ satisfies
the following inequality
\begin{equation*}
{\rm Re}\left\{1+\gamma\left(\frac{zf'(z)}{f(z)}-1\right)\right\}>0\quad (z\in \mathbb{D}),
\end{equation*}
then
\begin{equation*}
   \left|\left(\frac{f(z)}{z}\right)^2-1\right|<c\quad (z\in \mathbb{D}).
\end{equation*}
\end{corollary}

In order to prove next results, we need the following lemmas.
\begin{lemma}\label{MiMO}{\rm(}\cite{SMPM}{\rm)}
Let $q$ be univalent in the unit disk $\mathbb{D}$ and $\theta$
and $\phi$ be analytic in a domain $\mathbb{U}$ containing
$q(\mathbb{D})$ with $\phi(w) \neq0$ when $w\in q(\mathbb{D})$.
Set $Q(z) = zq'(z)\phi(q(z))$, $h(z) = \theta(q(z)) + Q(z)$.
Suppose that $Q$ is starlike (univalent) in $\mathbb{D}$, and
\begin{equation*}
  {\rm Re}\left\{\frac{zh'(z)}{Q(z)}\right\}={\rm Re}\left\{\frac{\theta'(q(z))}{\phi(q(z))}+\frac{zQ'(z)}{Q(z)}\right\}>0 \quad (z\in\mathbb{D}).
\end{equation*}
If $p$ is analytic in $\mathbb{D}$, with $p(0)=q(0)$, $p(\mathbb{D})\subset \mathbb{U}$ and
\begin{equation}\label{sub,mimo}
  \theta(p(z))+zp'(z)\phi(p(z))\prec \theta(q(z))+zq'(z)\phi(q(z)),
\end{equation}
then $p(z)\prec q(z)$, and q is the best dominant of \eqref{sub,mimo}.
\end{lemma}
\begin{lemma}\label{lem.MIM}
{\rm(}see \cite{MIMO}, see also \cite[p. 24]{MM}{\rm)} Assume that
$\mathcal{Q}$ is the set of analytic functions that are injective
on $\overline{\mathbb{D}}\backslash E(f)$, where $E(f):\{\omega:
\omega\in \partial \mathbb{D}~ and~\lim_{z\rightarrow \omega}
f(z)=\infty \}$, and are such that $f'(\omega)\neq 0$ for
$(\omega\in \partial \mathbb{D}\backslash E(f)$. Let $\psi\in
\mathcal{Q}$ with $\psi(0) =a$ and let $\varphi(z)=a+a_mz^m+\cdots$
be analytic in $\mathbb{D}$ with $\varphi(z)\not\equiv a$ and
$m\in \mathbb{N}$. If $\varphi\not\prec \psi$ in $\mathbb{D}$,
then there exist points $z_0=r_0e^{i\theta}\in\mathbb{D}$  and
$\omega_0\in
\partial \mathbb{D}\backslash E(\psi)$, for which $\varphi(|z| <r_0)\subset \psi(\mathbb{D})$,
$\varphi(z_0)=\psi(\omega_0)$ and $z_0\varphi'(z_0) =k \omega_0
\psi'(\omega_0)$, for some $k\geq m$.
\end{lemma}

\begin{theorem}
  Let $p$ be an analytic function on $\mathbb{D}$ and with $p(0)=1$ and let $c \in (0,1]$. If the function $p$ satisfies the subordination
  \begin{equation}\label{subt2.2}
    \frac{1}{3}p^3(z)+zp'(z)\prec \frac{1}{3}\left(\sqrt{1+cz}\right)^3+\frac{cz}{2\sqrt{1+cz}}\quad (z\in \mathbb{D}),
  \end{equation}
  then $p$ also satisfies the subordination
  \begin{equation*}
    p(z)\prec \sqrt{1+cz}\quad (z\in \mathbb{D}),
  \end{equation*}
  and $\sqrt{1+cz}$ is the best dominant of \eqref{subt2.2}.
\end{theorem}

\begin{proof}
  Set
  \begin{equation*}
    q_c(z)=\sqrt{1+cz},\quad \theta(w)=\frac{1}{3} w^3,\quad \phi(w)=1.
  \end{equation*}
  Then, $q_c$ is analytic and univalent in $\mathbb{D}$ and $q_c(0)=p(0)=1$. Moreover, the functions $\theta(w)$, and $\phi(w)$ are analytic with
$\phi(w)\neq 0$ in the $w$-plane. The function, also
\begin{equation*}
  Q(z)=zq'_c(z)\phi(q(z))=\frac{cz}{2\sqrt{1+cz}}=zq'_c(z),
\end{equation*}
  is starlike function. We now put
  \begin{equation}\label{h(z)}
    h(z)=\theta(q_c(z))+Q(z)=\frac{1}{3}q^3_c(z)+zq'_c(z),
  \end{equation}
  then we have
  \begin{equation*}
    {\rm Re}\left\{\frac{zh'(z)}{Q(z)}\right\}={\rm Re}\left\{1+cz+\left(1+\frac{zq''_c(z)}{q'_c(z)}\right)\right\}>1-c\geq0,
  \end{equation*}
  for all $z\in\mathbb{D}.$
  Then, the function $h$ given by \eqref{h(z)} is close-to-convex and univalent in $\mathbb{D}$. Therefore, by applying the Lemma \ref{MiMO}
  and \eqref{subt2.2}, we conclude that $p(z)\prec q_c(z)$ and $q_c(z)$ is the best dominant
of \eqref{subt2.2}. Thus the proof is completed.
\end{proof}
 Taking $p(z)=zf'(z)/f(z)$, we have the following result:

 \begin{corollary}
  Let $c \in (0,1]$. If a function $f$ satisfies the subordination
   \begin{equation*}\label{subcor}
     \frac{1}{3}\left(\frac{zf'(z)}{f(z)}\right)^3+\left(1+\frac{zf''(z)}{f'(z)}
     -\frac{zf'(z)}{f(z)}\right)\left(\frac{zf'(z)}{f(z)}\right)
     \prec
     \frac{1}{3}\left(\sqrt{1+cz}\right)^3+\frac{cz}{2\sqrt{1+cz}},
   \end{equation*}
   for all $z\in\mathbb{D}$, then $f\in\mathcal{S}^*(q_c)$.
 \end{corollary}

\begin{theorem}\label{t23}
Let $ k\geq1$ and let $ 0<c\leq1.$ If $p$ is an analytic function
in $\mathbb{D}$ with $p(0)=1$ and it  satisfies the condition
\begin{equation}\label{Re}
   {\rm Re}\left\{p(z)(p(z)+zp'(z))\right\}>1+c(1+k/2),
\end{equation}
  for all $z\in\mathbb{D}$ then
\begin{equation*}
    p(z)\prec \sqrt{1+cz}\quad (z\in\mathbb{D}).
\end{equation*}
\end{theorem}

\begin{proof}
Assume that $p(z)\not \prec q_c(z)=\sqrt{1+cz}$. Then there exist
points $z_0$, $|z_0| <1$ and $\omega_0$, $|\omega_0| =1$,
$\omega_0\neq 1$ that satisfy the following conditions
\begin{equation*}
    p(z_0)=q_c(\omega_0),\quad p(|z|<|z_0|)\subset q_c(\mathbb{D}),\quad  |\omega_0|=1.
\end{equation*}
From Lemma \ref{lem.MIM}, it follows that there exists a number
$k\geq1$ such that
\begin{equation}\label{eq}
\left\{p(z_0)(p(z_0)+z
p'(z_0))\right\}=\left\{q_c(\omega_0)(q_c(\omega_0)+k \omega_0
q'_c(\omega_0))\right\}=1+c(1+k/2)\omega_0.
\end{equation}
By setting $\omega_0=e^{i\theta}$, $\theta\in[-\pi,\pi]$ in
\eqref{eq}, we obtain
\begin{equation*}
{\rm Re}\{1+c(1+k/2)\omega_0\}=1+c(1+k/2)\cos\theta\leq
1+c(1+k/2).
\end{equation*}
But it contradicts our assumption \eqref{Re} and therefore $p(z)
\prec q_c(z)$ in $\mathbb{D}$.
\end{proof}
\begin{corollary}\label{c29}
Let $0<c\leq1$ and let $ k\geq1.$ If $f$ satisfies the following
inequality
\begin{equation*}
   {\rm Re}\left\{\left(\frac{zf'(z)}{f(z)}\right)^2\left(2+\frac{zf''(z)}{f'(z)}-\frac{zf'(z)}{f(z)}\right)\right\}
   >1+c(1+k/2),
\end{equation*}
  for all $z\in\mathbb{D}$, then $f\in\mathcal{S}^*(q_c)$.
\end{corollary}

\end{document}